\documentclass{amsart}
\usepackage{amsrefs}
\usepackage{hyperref}

\title{Quantum Functions}
\author{Andre Kornell}
\address{Department of Mathematics\\
University of California \\ Berkeley, CA 94720-3840}
\email{kornell@math.berkeley.edu}
\thanks{The research reported here was supported by National Science Foundation grant DMS-0753228.}

\newtheorem{theorem}{Theorem}[section]
\newtheorem{lemma}[theorem]{Lemma}
\newtheorem{proposition}[theorem]{Proposition}

\theoremstyle{definition}
\newtheorem{definition}[theorem]{Definition}

\theoremstyle{remark}

\theoremstyle{plain}
\newtheorem*{theorem*}{Theorem}
\newtheorem*{lemma*}{Lemma}
\newtheorem*{proposition*}{Proposition}

\theoremstyle{definition}
\newtheorem*{definition*}{Definition}

\theoremstyle{remark}
\newtheorem*{conventions}{Conventions used in this article}

\newcommand{\subsetof}{\subseteq}
\newcommand{\To}{\longrightarrow}
\newcommand{\Tensor}{\overline{\otimes}}
\newcommand{\tensor}{\otimes}

\newcommand{\union}{\cup}
\newcommand{\suchthat}{\,|\,}

\newcommand{\inv}{^{-1}}

\newcommand{\vN}{\mathrm{vN}}
\newcommand{\qF}{\mathbf{qF}}

\renewcommand{\H}{\mathcal H}

\newcommand{\B}{\mathcal B}
\newcommand{\M}{\mathcal M}
\newcommand{\N}{\mathcal N}
\newcommand{\V}{\mathcal V}
\newcommand{\W}{\mathcal W}
\newcommand{\K}{\mathcal K}
\newcommand{\F}{\mathcal F}

\renewcommand{\L}{\mathcal L}

\newcommand{\CC}{\mathbb C}

\newcommand{\GGG}{\mathfrak G}

\renewcommand{\vN}{\mathbf{vN}}

\begin{document}

\begin{abstract}
Weaver has recently defined the notion of a quantum relation on a von Neumann algebra. We demonstrate that the corresponding notion of a quantum function between two von Neumann algebras coincides with that of a normal unital $*$-homomorphism in the opposite direction. This is essentially a reformulation of a previously known result from the theory of Hilbert von Neumann modules.
\end{abstract}

\maketitle

A relation between sets $X$ and $Y$ is simply a subset of $Y \times X$. Motivated 
by Kuperberg and Weaver's work on quantum metrics \cite{KuperbergWeaver}, Weaver has recently proposed the following generalization of relations to the noncommutative setting:
\begin{definition*}[Weaver, \cite{Weaver}*{Definition 2.1}]
 Let $\M \subsetof \B(\H)$ and $\N \subsetof\B(\K)$ be von Neumann algebras. A \emph{quantum relation} between $\M$ and $\N$ is an ultraweakly closed subspace $\V \subsetof \B(\H, \K)$ such that $\N'\V\M' \subsetof \V$.
\end{definition*}

This definition reduces to the usual one when $\M = \ell^\infty(X)$ and $\N= \ell^\infty(Y)$. It has many other virtues. It is simple to state and simple to handle, and many familiar properties of relations have natural analogs. 

\begin{definition*}[Weaver, \cite{Weaver}*{Definition 2.4}]
Let $\M_0$, $\M_1$ and $\M_2$ be von Neumann algebras.
\begin{itemize}
 \item The \emph{diagonal quantum relation} on $\M_0$ is the quantum relation $\M_0'$ between $\M_0$ and $\M_0$.
\item If $\V$ is a quantum relation between $\M_0$ and $\M_1$, then the \emph{inverse} of $\V$ is the quantum relation $\V^*$ between $\M_1$ and $\M_0$.
\item If $\V_0$ is a quantum relation between $\M_0$ and $\M_1$, and $\V_1$ is a quantum relation between $\M_1$ and $\M_2$, then their \emph{composition} $\V_1\circ\V_0$  is the quantum relation between $\M_0$ and $\M_2$ defined by
$$\V_1\circ\V_0 = \V_1 \V_0 = \overline{\mathrm{span} \{v_1v_0 \suchthat v_1 \in \V_1, \, v_0 \in \V_0\}}^{ultraweak}.$$
\end{itemize}
\end{definition*}

If we interpret inclusion between quantum relations as the proper generalization of inclusion between classical relations, we arrive immediately at the following definitions.

\begin{definition*}[Weaver, \cite{Weaver}*{Definition 2.4}]
 Let $\V$ be a quantum relation on a von Neumann algebra $\M$. Then $\V$ is said to be
\begin{itemize}
\item \emph{reflexive} in case $\M' \subsetof \V$,
\item \emph{symmetric} in case $\V^*= \V$,
\item \emph{antisymmetric} in case $\V \cap \V^* \subsetof \M'$, and
\item \emph{transitive} in case $\V\V\subsetof \V$.
\end{itemize}
\end{definition*}

Thus, Weaver has generalized a large class of mathematical objects including orderings, graphs, equivalence relations, etc.

In foundations, a function from a set $X$ to a set $Y$ is typically defined as a relation $F \subsetof Y \times X$  such that for each element $x \in X$, there is exactly one element $y \in Y$ such that $(y,x) \in F$. Denoting the diagonal relations on $X$ and $Y$ by $\Delta_X$ and $\Delta_Y$ respectively, we may restate this condition as a pair of inequalities:
\begin{enumerate}
 \item $\Delta_X\subsetof F^* \circ F$
 \item $F \circ F^* \subsetof \Delta_Y$
\end{enumerate}
It is therefore natural to investigate the quantum relations that satisfy the analogs of these inequalities, i.e., quantum functions. In fact, we show that they correspond exactly to the normal unital $*$-homomorphisms:

\begin{definition*}
 Let $\M \subsetof \B(\H)$ and $\N \subsetof \B(\K)$ be von Neumann algebras. A \emph{quantum function} from $\M$ to $\N$ is a quantum relation $\V$ between $\M$ and $\N$ such that $\M' \subsetof \V^* \V$ and $\V\V^* \subsetof \N'$.
\end{definition*}

\begin{theorem*}
 Let $\M\subsetof \B(\H)$ and $\N \subsetof \B(\K)$ be von Neumann algebras. There is a canonical bijective correspondence between normal unital $*$-homomorphisms $\N \To \M$, and quantum functions from $\M$ to $\N$.  This correspondence is functorial.
\end{theorem*}

By chance, Debashish Goswami visited Berkeley soon after I posted this paper to arxiv.org. He explained that, although the motivation and the statement of theorem are new, the proof is essentially already known. He recommended his book \emph{Quantum Stochastic Processes and Noncommutative Geometry} \cite{SinhaGoswami} as a reference. Specifically, the core argument is essentially the proof of \cite{SinhaGoswami}*{Theorem 4.2.7}; the condition $\V \V^* \subsetof \N'$ implies that $\V^*$ is a right Hilbert $\N'$-module. However, this theorem cannot be applied directly because Hilbert von Neumann modules are defined to be closed in the strong operator topology, whereas a quantum relation is defined to be closed in the ultraweak topology. Later, Alexandru Chirvasitu pointed out that both theorems can be obtained from \cite{Rieffel}*{Proposition 6.12}. I thank Alexandru Chirvasitu and Debashish Goswami for their explanations. I also thank my advisor, Marc Rieffel, for suggesting Nik Weaver's papers to me.

\centerline{\rule[3 pt]{6 cm}{.1 pt}}

\begin{conventions}
 Let $\H$ be a Hilbert space. If $\xi \in \H$, then $\hat \xi \in \B(\CC, \H)$ is defined by $\hat \xi(c) = c\xi$. If $\V$ and $\W$ are ultraweakly closed subspaces of $\B(\H)$, then $\V\W = \overline{\mathrm{span}\{vw \suchthat v\in \V, \, w \in \W\}}^{uw}$, and $\V \Tensor \W = \overline{\mathrm{span}\{v\tensor w \suchthat v\in \V, \, w \in \W\}}^{uw} $. The tensor product of two Hilbert spaces is defined in such a way that $\H \tensor \CC = \H = \CC \tensor \H$. The von Neumann algebra of scalar operators on $\H$ is denoted by $\CC_\H$.
\end{conventions}

\begin{definition*}
 Let $\M$ and $\N$ be von Neumann algebras. Then $\mathbf{vN}(\N, \M)$ denotes the set of normal unital $*$-homomorphisms $\N \To \M$, and  $\mathbf{qF}(\M, \N)$ denotes the set of quantum functions from $\M$ to $\N$.
\end{definition*}

\section{Functions from Homomorphisms}

Let $\M \subsetof \B(\H)$ and $\N \subsetof \B(\K)$ be von Neumann algebras, and let $\pi: \N \To \M$ be a normal unital $*$-homomorphism. 

\begin{theorem}[Dixmier, \cite{TakesakiI}*{Theorem IV.5.5}]\label{representation by isometry}
 There is a Hilbert space $\L$ and an isometry $w \in  \B(\H, \K \tensor \L)$ such that $\pi(b) = w^*(b \tensor 1)w $.
\end{theorem}

Let $\L$ and $w$ be as in Theorem \ref{representation by isometry} above.

\begin{proposition}\label{interwines}
For all $b \in \N$, $(b \tensor 1) w = w \pi(b)$.
\end{proposition}

\begin{proof}
 For all $b \in \N$,
$$ww^* (b^* \tensor 1)( b \tensor 1) ww^* = w\pi(b^* b)w^* = w\pi(b^*)\pi(b)w^* = ww^* (b^* \tensor 1) ww^* (b \tensor 1) w w^*,$$
so $ww^* (b^* \tensor 1) (1-ww^*) (b \tensor 1) ww^* = 0$. Since $1-ww^*$ is a projection, we conclude that $(1-ww^*)(\N \Tensor \CC_\L) ww^* = 0$, i.e., $ww^* \in (\N \Tensor \CC_\L)'$. Therefore, for all $b\in \N$, $(b \tensor 1) w = (b \tensor 1) w w^* w = w w^* (b\tensor 1) w = w \pi(b)$.
\end{proof}

\begin{definition}\label{definition of G}
  The set $\GGG(\pi) = \{v \in \B(\H, \K) \suchthat \forall b \in \N \,\,bv = v \pi(b)\}$ is a quantum relation between $\M$ and $\N$.
\end{definition}

\begin{proposition}\label{cosurjectivtiy}
 $\M' \subsetof \GGG(\pi)^* \GGG(\pi) $
\end{proposition}

\begin{proof}
  Choose a basis $\{e_\alpha\}_{\alpha \in I}$ of $\L$. For each $\alpha \in I$, $(1 \tensor \hat e_\alpha ^*)w \in \B(\H, \K)$ satisfies $b(1 \tensor \hat e_\alpha^*)w = (1 \tensor \hat e_\alpha^*)(b \tensor 1)w = (1 \tensor \hat e_\alpha^*)w \pi(b)$
for all $b \in \N$, i.e.,  $(1 \tensor \hat e_\alpha^*)w \in \GGG(\pi)$. It follows that
$$1 = w^*w =  \sum_{\alpha \in I} w^* (1 \tensor \hat e_\alpha \hat e_\alpha^*)w = \sum_{\alpha \in I} \left((1 \tensor \hat e_\alpha ^*)w\right)^* \left((1 \tensor \hat e_\alpha ^*)w\right),$$
so $\CC_\H \subsetof \GGG(\pi)^*\GGG(\pi)$. We conclude that $\M' = \CC_\H \M' \subsetof \GGG(\pi)^*\GGG(\pi)\M' = \GGG(\pi)^*\GGG(\pi)$.
\end{proof}

\begin{proposition}\label{coinjectivity}
 $\GGG(\pi) \GGG(\pi)^* \subsetof \N'$
\end{proposition}

\begin{proof}
 For all $v_0, v_1 \in \GGG(\pi)$, and all $b \in \N$, $b v_0 v_1^* = v_0 \pi(b) v_1^* = v_0 (v_1 \pi(b^*))^* = v_0 v_1^* b$. 
\end{proof}

\begin{proposition}
 Therefore $\GGG(\pi)$ is a quantum function from $\M$ to $\N$.
 \end{proposition}

Thus, we have defined a function $\GGG: \vN(\N,\M) \To \qF(\M, \N)$.

\section{Homomorphisms from Functions}

 Let $\M\subseteq \B(\H)$ and $\N \subseteq \B(\K)$ be von Neumann algebras, and $\V$ a quantum function from $\M$ to $\N$. The following lemma is a special case of Paschke's structural theorem for self-dual Hilbert $W^*$-modules \cite{Paschke}*{Theorem 3.12}.

\begin{lemma}\label{family of partial isometries}
 There exists a family Lemma $\{u_\alpha\}_{\alpha \in I}$ of partial isometries in $\V$ such that
	\begin{enumerate}
	 \item for all distinct $\alpha,\beta \in I$, $u_\alpha u_\beta^* = 0$, and
	 \item $ \sum u_\alpha^* u_\alpha = 1$.
	\end{enumerate}

\end{lemma}

\begin{proof}
 Let $\F$ be the collection of all sets $S$ of partial isometries in $\V$ such that $u\tilde u^* = 0$ whenever $u,\tilde u \in S$ are distinct. Applying Zorn's Lemma, we obtain a maximal such set $\hat S$.

Suppose that $\sum_{u \in \hat S} u^*u \neq 1$, and let $p = 1 - \sum_{u \in \hat S} u^*u$. The subspace $\V p$ is non-zero because $1 \in \M' \subsetof \V^*\V $. Therefore, pick $v \neq 0$  in $\V p \subsetof \V$. We will now obtain a partial isometry to add to $\hat S$ from the polar decomposition of $v$.

Let $W^*(\V)$ be the von Neumann algebra generated by $\V$, which is the ultraweakly closed subspace of $\B(\H \oplus \K)$ generated by finite words in the elements of $\V$ and their conjugates. Since $\V \V^* \V \subsetof \N' \V \subsetof \V$, $W^*(\V)$ is in fact generated by words of one of the following forms: $1$, $v_0$, $v_0^*$, $v_1^* v_0$, and $v_1 v_0^*$. We conclude that $[\K]W^*(\V)[\H] = \V$, where $[\H]$ and $[\K]$ denote projections onto $\H$ and $\K$ respectively.

Let $v = u_v |v|$ be the polar decomposition of $x$. Since $[\K]v[\H] = v$, $[\K]u_v[\H]= u_v$, so $u_v \in \V$. Furthermore, since $vp = v$, $u_v p = u_v$, so $u_v u^* = 0 $ for all $u \in \hat S$. Thus, $\hat S \union \{u_v\}$ is an element of $\F$ strictly larger than $\hat S$, a contradiction. 
\end{proof}

\begin{definition}\label{definition of G inverse}
 Let $\GGG\inv(\V)$ be the normal unital $*$-homomorphism defined by   $$\GGG\inv(\V)(b) = \sum_{\alpha \in I}  u_\alpha^* b u_\alpha = w_I^*(b \tensor 1) w_I,$$ where $\{u_\alpha\}_{\alpha \in I}$ is any family of partial isometries in $\V$ such that $\sum u_\alpha^* u_\alpha = 1$ and $u_\alpha u_\beta^* = 0$ whenever $\alpha \neq \beta$, and the isometry $w_I \in B(\H, \K \tensor \ell^2(I))$ is defined by $w_I\xi = \sum_{\alpha \in I} u_\alpha \xi \tensor e_\alpha$.
\end{definition}

\begin{proposition}\label{G inverse is well defined}
 The normal unital $*$-homomorphism $\GGG\inv(\V)$ is well defined.
\end{proposition}

\begin{proof}
 Apply Lemma \ref{family of partial isometries} above to obtain a family $\{u_\alpha\}_{\alpha\in I}$ of partial isometries in $\V$ such that  $\sum u_\alpha^* u_\alpha = 1$ and $u_\alpha u_\beta^* = 0$ whenever $\alpha \neq \beta$. Let $w_I = \sum u_\alpha \tensor \hat e_\alpha \in \B(\H, \K \tensor \ell^2(I))$. For all $\xi \in \H$,
$$\|w_I(\xi)\|^2 = \sum_{\alpha, \beta} \langle u_\alpha (\xi) \tensor e_\alpha | u_\beta(\xi) \tensor e_\beta \rangle = \sum _\alpha \langle u_\alpha \xi | u_\alpha \xi\rangle = \langle \xi | \left(\sum_\alpha u_\alpha^* u_\alpha \right) \xi \rangle = \|\xi\|^2.$$ Thus, $w_I$ is an isometry, and we may now define a normal unital completely positive map $\pi_I :\N \To \B(\H)$ by $\pi_I(b) = w_I^* (b \tensor 1) w_I = \sum u_\alpha^* b u_\alpha$. For all $b_0, b_1 \in \N$,
\begin{align*} \pi_I(b_0) \pi_I(b_1) &= \left(\sum_\alpha u_\alpha^* b_0 u_\alpha \right) \left(\sum_\beta u_\beta^* b_1 u_\beta \right)  \\ &= \sum_\alpha u_\alpha^* b_0 u_\alpha u_\alpha^* b_1 u_\alpha \\ & =
\sum_\alpha u_\alpha^* b_0 b_1 u_\alpha u_\alpha^*  u_\alpha
\\ & =  \sum_\alpha u_\alpha^* b_0 b_1   u_\alpha = \pi_I(b_0b_1)
\end{align*}
because $u_\alpha u_\alpha^* \in \V \V^* \subsetof \N'$. We conclude that $\pi_I$ is a normal unital $*$-homomorphism.
For all $b \in \N$, $c \in \M'$,

\begin{align*}
 c \pi_I(b)  & = \sum_{\alpha, \beta \in I} u_\alpha^* u_\alpha c u_\beta^* b u_\beta 
 =  \sum_{\alpha, \beta \in I} u_\alpha^*  b u_\alpha c u_\beta^* u_\beta
= \pi_I(b)c
\end{align*}
because $u_\alpha c u_\beta^* \in\V \M' \V^* \subseteq \V \V^* \subseteq \N'$. By the Double Commutant Theorem,  $\pi_I(\N) \subseteq \M$, so $\pi_I$ may be viewed as a normal unital $*$-homomorphism $\N \To \M$.

If $\{u_\alpha\}_{\alpha\in J}$ is another family of partial isometries that satisfies $\sum_{\alpha\in J} u_\alpha^* u_\alpha = 1 $ and $u_\alpha u_\beta^* = 0 $ for distinct $\alpha, \beta \in J$, we may obtain in the same way a normal unital $*$-homomorphism $\pi_J: \N \To \M$. However, for all $b \in \N$,
$$\pi_J(b) = \sum_{\alpha \in I, \beta \in J } u_\alpha^* u_\alpha u_\beta^* b u_\beta =  \sum_{\alpha \in I, \beta \in J } u_\alpha^*  b u_\alpha u_\beta^* u_\beta = \pi_I (b)$$
because $u_\alpha u_\beta^* \in \V\V^* \subsetof \N'.$ Thus, $\pi_I$ is independent of our choice of family $\{u_\alpha\}_{\alpha \in I}$, and we may define $\GGG\inv(\V) = \pi_I$.
\end{proof}

Thus, we have defined a function $\GGG\inv: \qF(\M, \N) \To \vN(\N, \M)$.

\section{$\GGG\inv$ is the Inverse of $\GGG$}

Let $\M \subsetof \B(\H)$ and $\N \subsetof \B(\K)$ be von Neumann algebras.

\begin{proposition}\label{G inverse is the left inverse}
 Let $\pi: \N \To \M$ be a normal unital $*$-homomorphism. Then $\GGG\inv(\GGG(\pi)) = \pi$.
\end{proposition}

\begin{proof}
 Let $\{u_\alpha\}_{\alpha \in I}$ be a family of partial isometries in $\GGG(\pi)$ such that $\sum_{\alpha \in I} u^*_\alpha u_\alpha =1$ and $u_\alpha u_\beta^* = 0$ whenever $\alpha \neq \beta$. For all $b \in \N$,
$$\GGG\inv(\GGG(\pi))(b) = \sum_{\alpha \in I} u_\alpha^* b u_\alpha = \sum_{\alpha \in I} u_\alpha^* u_\alpha \pi(b) = \pi(b).$$
\end{proof}

\begin{proposition}\label{homotopy}
 Let $\L$ be a Hilbert space, and let $w_0, w_1 \in \B(\H , \K \tensor \L)$ be isometries such that $\pi_i(b) = w_i^* (b \tensor 1) w_i$ defines a pair of $*$-homomorphisms $\N \To \M$. If $\pi_0 = \pi_1$, then $w_0 w_1^* \in (\N \Tensor \CC_\L)'$.
\end{proposition}

\begin{proof}
 For all $b \in \N$, $(b \tensor 1) w_0 w_1^* = w_0 \pi_0(b) w_1^* = w_0 \pi_1(b) w_1^* = w_0(w_1\pi_1(b^*))^* = w_0 ((b^* \tensor 1 )w_1))^* = w_0 w_1^* (b \tensor 1)$.
\end{proof}

\begin{lemma}\label{w I generates V}
 Let $\V$ be a quantum function from $\M$ to $\N$, and let $\{u_\alpha\}_{u \in I}$ and $w_I$ be as in Definition \ref{definition of G inverse}. 
Then $(\N \Tensor \CC_{\ell^2(I)})'w_I \M' = \V \Tensor \B(\CC, \ell^2(I))$.
\end{lemma}

\begin{proof}
 By definition, $w_I = \sum u_\alpha \tensor \hat e_\alpha \in \V \Tensor \B(\CC, \ell^2(I))$, so $(\N \Tensor \CC_{\ell^2(I)})'w_I \M' = (\N' \Tensor \B(\ell^2(I)))w_I \M' \subseteq \V \Tensor \B(\CC, \ell^2(I))$.

Let $f \in \ell^2(I)$ and $v \in \V$. Then for all $\alpha \in I$,
$$
u_\alpha \tensor \hat f
=
\sum_{\beta \in I} (1 \tensor \hat f \hat e_\alpha^*)(u_\beta \tensor \hat e_\beta)
=
(1 \tensor \hat f \hat e_\alpha^*)w_I
\in
  (\N \Tensor \CC_{\ell^2(I)})'w_I \M',$$
so
$$
v \tensor \hat f
=
\sum_{\alpha \in I} (v \tensor \hat f) (u_\alpha^* u_\alpha)
=
\sum_{\alpha \in I} (v u_\alpha^* \tensor 1) (u_\alpha \tensor \hat f)
\in 
(\N \Tensor \CC_{\ell^2(I)})'w_I \M'
$$ 
because $vu_\alpha^* \in \V \V^* \subsetof \N'$. It follows that $\V \Tensor \B(\CC, \ell^2(I)) \subsetof (\N \Tensor \CC_{\ell^2(I)})'w_I \M',$ concluding the proof.
\end{proof}

\begin{proposition}\label{G inverse is injective}
 The function $\GGG\inv$ is injective.
\end{proposition}

\begin{proof}
 For $k \in \{0,1\}$, let $\V_k \in \mathbf{qF(\M, \N)}$, and let $\{u_\alpha\}_{\alpha \in I_k}$ be a family of partial isometries in $\V_k$ such that $u_\alpha u_\beta^* =0$ for distinct $\alpha,\beta \in I_k$, and $\sum u_\alpha^* u_\alpha = 1$. We may assume that $I_0$ and $I_1$ have equal, non-zero cardinality by throwing in indexed instances of the zero partial isometry where necessary. Thus, we may choose a unitary $s \in \B (\ell^2(I_0), \ell^2 (I_1))$.

Suppose that $\GGG\inv(\V_0) = \GGG\inv(\V_1)$. For all $b \in \N$,
\begin{align*}
 ((1\tensor s) w_{I_0})^* (b \tensor 1) ((1 \tensor s) w_{I_0}) = w_{I_0}^* (b \tensor 1) w_{I_0}  & = \GGG\inv(\V_0)(b) \\ &=  \GGG\inv(\V_1) (b) = w_{I_1}^* (b\tensor 1) w_{I_1}
\end{align*}
By Proposition \ref{homotopy}, $ (1 \tensor s)w_{I_0}w_{I_1}^* \in (\N \Tensor \CC_{\ell^2(I_1)})'$.

By Lemma \ref{w I generates V} above, 
\begin{align*}
\V_1 \Tensor \B(\CC, \ell^2(I_1))
&= 
(\N \Tensor \CC_{\ell^2(I_1)})'w_{I_1} \M'
\\ & \supseteq
 (\N \Tensor \CC_{\ell^2(I_1)})'((1 \tensor s)w_{I_0}w_{I_1}^*) w_{I_1} \M'
\\ & =
(\N \Tensor \CC_{\ell^2(I_1)})'(1 \tensor s)w_{I_0} \M'
\\ & =
(1 \tensor s)(\N \Tensor \CC_{\ell^2(I_0)})'w_{I_0} \M'
\\ &=
(1 \tensor s)(\V_0 \Tensor \B(\CC_, \ell^2(I_0)))
\\ & = 
\V_0 \Tensor \B(\CC, \ell^2(I_1)).
\end{align*}
Choosing an arbitrary unit vector $f \in \ell^2(I_1)$, we conclude that
$$\V_1 = (1 \tensor \hat f^*) \left(\V_1 \Tensor \B(\CC, \ell^2(I_1))\right)
\supseteq 
(1 \tensor \hat f^*) \left(\V_0 \Tensor \B(\CC, \ell^2(I_1))\right)
= \V_0.
 $$
Similarly, $\V_1 \subsetof \V_0$, so $\V_1 = \V_0$.
\end{proof}

\begin{theorem}\label{G is bijective}
 Let $\M\subsetof \B(\H)$ and $\N \subsetof \B(\K)$ be von Neumann algebras. The function 
$$\GGG: \mathbf{vN}(\N, \M) \To \mathbf{qF}(\M, \N)$$
defined by
$$\GGG(\pi) = \{v \in \B(\H, \K)\suchthat \forall b \in \N \,\, bv = \pi(b)v\}$$
is a bijection. 
\end{theorem}

\begin{proof}
 The theorem follows immediately from Propositions \ref{G inverse is the left inverse} and \ref{G inverse is injective}.
\end{proof}

\section{Functoriality}

\begin{proposition}
 Let $\M_0$, $\M_1$, and $\M_2$ be von Neumann algebras. If $\V_0 \in \qF(\M_0,\M_1)$ and $\V_1 \in \qF(\M_1, \M_2)$, then $\V_1\V_0 \in \qF(\M_0, \M_2)$. 
\end{proposition}

\begin{proof}
$$\M_0' \subsetof \V_0^* \V_0 \subsetof \V_0^* \M_1' \V_0 \subsetof \V_0^* \V_1^* \V_1 \V_0 = (\V_1 \V_0)^*(\V_1 \V_0)$$ $$(\V_1 \V_0)(\V_1 \V_0)^* = \V_1 \V_0 \V_0^* \V_1 \subsetof \V_1 \M_1' \V_1^* \subsetof \V_1 \V_1^* \subsetof \M_2'$$
\end{proof}

\begin{proposition}\label{G respects composition}
 Let $\M_0$, $\M_1$, and $\M_2$ be von Neumann algebras. If  $\pi_1: \M_2 \To \M_1$ and $\pi_0: \M_1 \To \M_0$ are normal unital $*$-homomorphisms, then $\GGG(\pi_0 \circ \pi_1) = \GGG(\pi_1) \GGG(\pi_0)$. 
\end{proposition}

\begin{proof}
 Clearly, $\GGG(\pi_1) \GGG(\pi_0) \subsetof \GGG(\pi_0 \circ \pi_1)$. By Definition \ref{definition of G inverse}, $\GGG\inv(\GGG(\pi_1)\GGG(\pi_0)) =  \GGG\inv(\GGG(\pi_0 \circ \pi_1))$. By Theorem \ref{G is bijective}, we conclude that $\GGG(\pi_1)\GGG(\pi_0)= \GGG(\pi_0 \circ \pi_1)$.
\end{proof}

\begin{proposition}\label{G respects identity}
 Let $\M$ be a von Neumann algebra, and let $\iota: \M \To \M$ be the identity $*$-homomorphism. Then $\GGG(\iota) = \M'$.
\end{proposition}

\begin{proof}
 This is an immediate consequence of Definition \ref{definition of G}.
\end{proof}

\begin{definition}\label{definition of vN}
 Let $\mathbf{vN}$ be the category whose objects are von Neumann algebras, and whose morphisms are normal unital $*$-homomorphisms.
\end{definition}

\begin{definition}\label{definition of qM}
 Let $\mathbf{qF}$ be the following category:
\begin{itemize}
 \item The objects of $\mathbf{qF}$ are von Neumann algebras.
 \item For any two objects $\M$ and $\N$, a morphism from $\M$ to $\N$ is a quantum function from $\M$ to $\N$.
 \item For any two morphisms $\V_0 \in \mathbf{qF}(\M_0, \M_1)$ and $\V_1 \in \mathbf{qF}(\M_1, \M_2)$, $\V_1 \circ \V_0 = \V_1 \V_0$.
 \item For any object $\M$, the identity morphism at $\M$ is the quantum function $\M' \in \mathbf{qF}(\M, \M)$.
\end{itemize}
\end{definition}

\begin{theorem}\label{G is a coisomorphism}
 The functor $\GGG: \mathbf{vN} \To \mathbf{qF}$ defined by 
\begin{itemize}
 \item for all objects $\M$ of $\mathbf{vN}$, $\GGG(\M) = \M$, and
 \item for all morphisms $\pi \in \mathbf{vN}(\N \subsetof \B(\K), \M \subsetof \B(\H))$, $$\GGG(\pi) = \{v \in \B(\H, \K) \suchthat \forall b \in \N\,\, bv=v\pi(b)\},$$
\end{itemize}
is a coisomorphism of categories.

\end{theorem}

\begin{proof}
 This is a straightforward consequence of Proposition \ref{G respects composition}, Proposition \ref{G respects identity}, and Theorem \ref{G is bijective}.
\end{proof}

\end{document}